\documentclass[a4paper,10pt]{amsart}
\usepackage[utf8]{inputenc}
\usepackage{hyperref}
\usepackage{ifluatex}
\usepackage{mathrsfs}
\usepackage{tikz}
  \usetikzlibrary{arrows,%
    decorations.markings,%
    matrix,%
    shapes}
\usepackage{tikz-cd}
\usepackage{wasysym}
\usepackage{xspace}
\usepackage{amsmath}	
\usepackage{amssymb}

\ifluatex
  \usepackage{fontspec}
  \usepackage{unicode-math}
 \else
  \usepackage{fourier}
\fi

\hypersetup{pdftitle={Classifying subcategories in quotients of exact categories},%
  pdfauthor={Emilie Arentz-Hansen},%
}

\DeclareMathAlphabet{\mathcal}{OMS}{cmsy}{m}{n}

\DeclareMathOperator{\Hom}{Hom}
\DeclareMathOperator{\Image}{Im}

\DeclareMathOperator{\obj}{obj}

\DeclareMathOperator{\Ext}{Ext}
\DeclareMathOperator{\depth}{depth}
\DeclareMathOperator{\mmod}{mod}

\DeclareMathOperator{\CM}{CM}
\DeclareMathOperator{\pro}{proj}
\DeclareMathOperator{\Gpro}{Gproj}
\DeclareMathOperator{\injj}{inj}

\newcommand{\proj}{\ensuremath{\pro}\xspace}
\newcommand{\inj}{\ensuremath{\injj}\xspace}
\newcommand{\Gproj}{\ensuremath{\Gpro}\xspace}
\newcommand{\m}{\ensuremath{\mathrm{m}}\xspace}
\newcommand{\uHom}{\underline{\Hom}}
\newcommand{\A}{\ensuremath{\mathscr{A}}\xspace}

\newcommand{\E}{\ensuremath{\mathscr{E}}\xspace}
\newcommand{\D}{\ensuremath{\mathscr{D}}\xspace}

\newcommand{\EN}{\ensuremath{\E/\N}\xspace}
\newcommand{\uE}{\ensuremath{\E/\N}\xspace}
\newcommand{\uEm}{\ensuremath{\D_\N}\xspace}
\newcommand{\F}{\ensuremath{\mathscr{E}}\xspace}

\newcommand{\uF}{\ensuremath{\underline{\F}}\xspace}
\newcommand{\uFm}{\ensuremath{\underline{\D}}\xspace}
\newcommand{\GR}{\ensuremath{\mathcal{G}(R)}\xspace}
\newcommand{\modR}{\ensuremath{\mmod R}\xspace}
\newcommand{\N}{\ensuremath{\mathscr{N}}\xspace}
\newcommand{\T}{\ensuremath{\mathscr{T}}\xspace}

\newcommand{\calI}{\ensuremath{\mathcal{I}}\xspace}
\newcommand{\calS}{\ensuremath{\mathcal{S}}\xspace}
\newcommand{\calSm}{\ensuremath{\mathcal{S}'}\xspace}
\newcommand{\mono}{monomorphism\xspace}

\newcommand{\epi}{epimorphism\xspace}
\newcommand{\adm}{admissible\xspace}
\newcommand{\amono}{\adm monomorphism\xspace}
\newcommand{\aepi}{\adm epimorphism\xspace}

\newcommand{\PO}{pushout\xspace}
\newcommand{\PB}{pullback\xspace}

\newcommand{\udl}{\underline}

\theoremstyle{definition}
\newtheorem{definition}[subsection]{Definition}
\newtheorem{remark}[subsection]{Remark}
\newtheorem{remarks}[subsection]{Remark}

\newtheorem{examples}[subsection]{Example}

\theoremstyle{plain}
\newtheorem{theorem}[subsection]{Theorem}
\newtheorem{proposition}[subsection]{Proposition}
\newtheorem{lemma}[subsection]{Lemma}
\newtheorem{corollary}[subsection]{Corollary}

\tikzset{middlearrow/.style={
    decoration={markings,
      mark= at position 0.2 with {\arrow{#1}} ,
    },
    postaction={decorate}
  }
}

\newcommand{\defi}[1]{\begin{definition}#1\end{definition}}

\newcommand{\bmtx}[1]{\begin{bmatrix}#1\end{bmatrix}}
\newcommand{\bmtxt}[1]{\Bigl[ \begin{smallmatrix}#1 \end{smallmatrix} \Bigr]}
\newcommand{\bmtxo}[1]{\bigl[ \begin{smallmatrix}#1 \end{smallmatrix} \bigr]}

\newcommand{\tikzcs}[1]{\begin{tikzcd}[ampersand replacement=\&,cramped, sep=scriptsize]#1\end{tikzcd}}

\title{Classifying subcategories in quotients of exact categories}
\author{Emilie Arentz-Hansen}
\address{Department of Mathematical Sciences, NTNU, NO-7491
  Trondheim, Norway}
\email{emilie\_ah@hotmail.com}

\begin{document}

\begin{abstract}
We classify certain subcategories in quotients of exact categories. In particular, we classify the triangulated and thick subcategories of an algebraic triangulated category, i.e. the stable category of a Frobenius category.
\end{abstract}

\subjclass[2010]{18E10, 18E30}
\keywords{Exact categories, Frobenius categories, algebraic triangulated categories}
\thanks{Part of this work was done while I visited Ryo Takahashi at Nagoya University. I would like to thank him for his hospitality and guidance.}

\maketitle

\section{Introduction}

Exact categories were introduced by Quillen in 1972, and have since played an important part in various areas of mathematics. They serve as a categorification of large parts of classical homological algebra. For example, exact sequences, projective and injective objects are natural notions within this framework.

It turns out that exact categories are intimately related to triangulated categories, via the so-called Frobenius categories. These are exact categories with enough projective and injective objects, and in which the two classes of objects coincide. Given such a category, one may form its stable category, which is triangulated in a very natural way. The triangulated categories that arise in this way are called algebraic, and the ones that appear naturally in homological algebra (homotopy categories of complexes, derived categories) are all of this form.

Given a triangulated category, one may ask for a classification of its thick or just triangulated subcategories. In this paper, we provide such a classification for the algebraic triangulated categories, in terms of certain subcategories in the ambient exact categories. In fact, we classify subcategories in more general quotients of exact categories, not just stable categories of Frobenius categories. 

\section{Preliminaries}

We recall the notions of exact categories, quotient categories and Frobenius categories. The reader should consult \cite{Buhler}, \cite{Happel} and \cite{HolZat} for a thorough introduction.

\begin{definition}
Let \E be an additive category and $A \xrightarrow{f} B \xrightarrow{g} C$ a sequence in \E. We call $(f, g)$ a {\bf kernel-cokernel pair} if $f$ is a kernel of $g$ and $g$ is a cokernel of $f$. Let \calS be a family of kernel-cokernel pairs in \E. If $(f, g) \in \calS$, then we call $f$ an {\bf admissible monomorphism} and $g$ an {\bf admissible epimorphism}. 
\end{definition}

We use the notations \tikzcs{A \arrow[r, "f", tail]  \&  B} and \tikzcs{B \arrow[r, "g", two heads]  \&  C} to specify that $f$ is an admissible monomorphism and $g$ is an admissible epimorphism, respectively.

\begin{definition} \label{def ex cat}
Let \E  be an additive category and $\mathcal{S}$ a family of kernel-cokernel pairs in \E. Assume that \calS is closed under isomorphisms and satisfies the following:
\begin{itemize}
\item[Ex0\phantom{$^{op}$}] $1_0$ is an \aepi, where 0 denotes the zero object.
\item[Ex1\phantom{$^{op}$}] The class of admissible epimorphisms is closed under composition.

\item[Ex2\phantom{$^{op}$}]  Admissible epimorphisms are stable under pullback along arbitrary morphisms: 
\begin{displaymath}
  \begin{tikzcd}
   B' \arrow[r, "g'", dashed, two heads]  \arrow[d, "h'", dashed]  \arrow[dr, phantom, "\scriptscriptstyle PB" ] &  C' \arrow[d, "h"]\\
		B \arrow[r, "g", two heads]  &  C
  \end{tikzcd}
\end{displaymath}

\item[Ex2$^{op}$] Admissible monomorphisms are stable under pushout along arbitrary morphisms: 
\begin{displaymath}
  \begin{tikzcd}
   A \arrow[r, "f", tail]  \arrow[d, "h"] \arrow[dr, phantom, "\scriptscriptstyle PO" ] &  B \arrow[d, "h'", dashed]\\
		A' \arrow[r, "f'", dashed, tail]  &  B'
  \end{tikzcd}
\end{displaymath}
\end{itemize}
In this case \calS is  an {\bf exact structure} on \E and the pair $(\E, \calS)$ is an {\bf exact category}. The elements of \calS are called {\bf short exact sequences}.
\end{definition}

From now on we fix an exact category $(\E, \calS)$. 

\begin{remarks} \label{rem ex cat}
{\bf (1) } This definition is equivalent to the classical definition given by Quillen in \cite{Quillen}: Keller proves this in \cite[Appendix A]{Keller}. Quillen's definition gives Ex1$^{op}$, i.e. that the class of admissible monomorphisms is closed under composition, and that for all $A, B$ in \E the sequence 
\begin{tikzcd}[ampersand replacement=\&, cramped, sep= 2.2 em]
	A \ar[r, "\bmtxt{1 \\ 0}", tail] \& A \oplus B \ar[r, "\bmtxo{0 & 1}", two heads] \& B
\end{tikzcd}
is short exact.

{\bf (2) } Any isomorphism is both an \amono and an \aepi.

{\bf (3) } 
An \aepi is in particular an \epi since it is a cokernel, and an \amono is a \mono since it is a kernel. 
\end{remarks}

Recall that a subcategory \D of 
\E is {\bf extension closed} if whenever 
	\tikzcs{X \arrow[r, "f", tail] \&  Y \arrow[r, "g", two heads]  \&  Z}
 is in \calS with $X, Z \in \D$, then $Y \in \D$. The following result shows that such subcategories are exact themselves.

\begin{proposition} \label{prop:exact subcat is exact}
Let \D be a full subcategory of 
\E with $0\in \D$. Assume that \D is extension closed and define
	$$\calSm:= \left\{\tikzcs{ X \arrow[r, "f", tail] \&  Y \arrow[r, "g", two heads]  \&  Z}  \in \calS \ \ \middle| \ \ X,Y,Z \in \D \right\}.$$ 
Then $(\D, \calS')$ is an exact category and $\D$ closed under isomorphisms. 
\end{proposition}

\begin{proof}
{\it \D is additive:} Since \D is a full subcategory containing 0 of an an additive category, \D is preadditive. For any objects $X,Y$ in \D, 
	\tikzcs{X \arrow[r, tail] \&  X\oplus Y \arrow[r, two heads]  \&  Y  } 
is short exact. Hence $X \oplus Y \in \D$ since \D is extension closed. Thus \D is additive.

{\it Ex2$^{op}$:} Let \tikzcs{A \arrow[r, "f", tail] \&  B \arrow[r, "g", two heads] \& C} be in \calSm and let $h:A\to D$ be in \D. Since \E is exact, the \PO along $f$ and $h$ exists, thus we get the following commutative diagram in \E:
\begin{displaymath}
  \begin{tikzcd}
 	A \arrow[r, "f", tail]  \arrow[d, "h"] \arrow[rd, phantom, "\scriptscriptstyle  PO"] &  B \arrow[r, "g", two heads] \arrow[d, "h'", dashed] &  C  \arrow[d, equal] \\
	D \arrow[r, "f'", dashed, tail]  &  P \arrow[r, "g'", dashed, two heads] &  C  
  \end{tikzcd}
\end{displaymath}
By \cite[Proposition 2.12]{Buhler} the second row is in \calS. Since \D is extension closed and $D, C$ are in \D, so is P. Thus the \PO along $f$ and $h$ lies in \D. Ex2 is done dually. 

{\it Ex1:} Let \tikzcs{A \arrow[r, "f", tail] \&  B \arrow[r, "g", two heads] \& C} and \tikzcs{P \arrow[r, "p", tail] \&  B' \arrow[r, "h", two heads] \& B} be in \calSm. Note that $gh$ is an \aepi in $(\E, \calS)$ since the category is exact. Hence $gh$ is an \aepi in $(\D, \calSm)$  as well if it has a kernel which lies in \D. By Ex2 there exists a \PB along $h$ and $f$ with objects in \D:
\begin{displaymath}
  \begin{tikzcd}
	A' \arrow[r, "f'"]  \arrow[d, "h'"]  \arrow[dr, phantom, "\scriptscriptstyle PB" ] &  B' \arrow[r, "gh"] \arrow[d, "h", two heads]  &  C  \arrow[d, equal]\\
	A \arrow[r, "f", tail]  &  B \arrow[r, "g", two heads] &  C
  \end{tikzcd}
\end{displaymath}
It is easy to show that $f'$ is the kernel of $gh$, hence $(f',gh)\in\calSm$.

{\it Closed under isomorphisms:} Let $f: X \to Y$ be an isomorphism and assume that $X$ lies in \D. Then 
	\tikzcs{X \arrow[r, "f", tail] \&  Y \arrow[r, "", two heads]  \&  0 }
is in \calS, hence $Y\in \D$ since \D is extension closed.
\end{proof}

We say that the exact structure on \D is {\bf induced by the exact structure on \E}. A typical example occurs when \E is an abelian category (with the standard exact structure) and \D is a full, extension closed subcategory containing the zero object. Indeed, for essentially small categories, this is equivalent to the definition of exact categories.


Next we consider quotient categories. Recall first that a collection \calI of morphisms in a preadditive category is a (two-sided) {\bf ideal} if
\begin{itemize}
	\item[(i)] $\calI(X, Y) := \calI \cap \Hom(X,Y)$ is a subgroup of the abelian group $\Hom(X,Y)$
	, and
	\item[(ii)] whenever $f \in \Hom(X,Y)$, $g \in \calI(Y, Z)$, $h \in \Hom(Z, W)$, then $hgf \in \calI(X,W)$.
\end{itemize}

\begin{definition}
Let \A be a preadditive category and \calI an ideal of \A. We define the {\bf quotient category} $\A / \calI$ by
\begin{samepage}
	$$\obj \A / \calI := \obj \A$$
	$$\Hom_{\A / \calI} (X,Y) := \Hom_\A (X,Y) / \calI(X,Y).$$
\end{samepage}
This means that $\A / \calI$  has the same objects as \A and that the morphisms in $\A / \calI$ are the equivalence classes of the morphisms in \A (two morphisms are equivalent if their difference is in \calI). The equivalence class of a morphism $f: X \to Y$ in \A  will be denoted by $\underline f$ in $\A / \calI$. 
When it is clear from the context, the notation $\uHom_\A(X,Y)$, or simply $\uHom(X,Y)$, will be used instead of $\Hom_{\A / \calI} (X,Y)$. 
\end{definition}

Composition in the quotient category $\A / \calI$ is well defined, and associativity follows from that in \A. Thus $\A / \calI$ is indeed a category. 

\begin{proposition}
Let \A be a (pre)additive category and \calI an ideal of \A. Then the quotient category $\A / \calI$ is also (pre)additive.
\end{proposition}

\begin{proof}
Since $\calI(X,Y)$ is a subgroup of the abelian group $\Hom(X,Y)$, the factor group 
	$$
	\uHom(X,Y) = \Hom_\A (X,Y) / \calI(X,Y)$$ 
is also an abelian group. The bilinearity of composition in $\A / \calI$ follows directly from the bilinearity of composition in \A. If 0 is a zero object in \A, then it is a zero object in $\A / \calI$ as well since 
$\uHom(0,X)$ and $\uHom(X,0)$ will contain only one morphism. If a biproduct of $X, Y$ in \A is given by
	$$X \overset{i_1}{\underset{p_1}{\rightleftarrows}} X \oplus Y  \overset{i_2}{\underset{p_2}{\leftrightarrows}} Y$$
then it is trivial to see that a biproduct of $X, Y$ in $\A / \calI$ is given by
	\begin{equation*}
	X \overset{\udl{i_1}}{\underset{\udl{p_1}}{\rightleftarrows}} X \oplus Y  \overset{\udl{i_2}}{\underset{\udl{p_2}}{\leftrightarrows}} Y.
\qedhere
\end{equation*}
\end{proof}

Now let \A be a preadditive category and $\N$ a (full) subcategory which is closed under finite direct sums. Define $\calI(X,Y) \subseteq \Hom(X,Y)$ to the collection of be all morphisms which factor through some object in $\N$, and denote by \calI the union of all $\calI(X,Y)$. Then \calI is an ideal of \A, and we denote the corresponding quotient category $\A / \calI$ by $\A / \N$.


The quotient categories we typically have in mind are the stable categories of Frobenius categories. Recall that an object $P$ in \E  is {\bf \calS-projective} if for all admissible epimorphisms \tikzcs{g:Y\arrow[r, two heads]\& Z} and for all morphisms $a: P \to Z$ there exists a (not necessarily unique) morphism $b: P\to Y$ such that $a = gb$: 
\begin{displaymath}
  \begin{tikzcd}
   & P \arrow[rd, "a"]  \arrow[d, "\exists b"', dashed]  &&  \\
	X \arrow[r, "f", tail] & Y \arrow[r, "g", two heads, near start]  &  Z & \in \calS
  \end{tikzcd}
\end{displaymath}
Dually, an object $I$ is {\bf \calS-injective} if for all admissible monomorphisms \tikzcs{f:X\arrow[r, tail]\& Y} and for all morphisms $a: X \to I$ there exists a morphism $b: Y\to I$ such that $a = bf$: 
\begin{displaymath}
  \begin{tikzcd}
		X \arrow[r, " f", tail] \arrow[rd, "a"'] &  Y \arrow[d, "\exists b", dashed]  \arrow[r, "g", two heads]  &  Z & \in \calS\\
      & I &&
  \end{tikzcd}
\end{displaymath}

Furthermore, \E has {\bf enough \calS-projectives} if for all objects $X$ there exists an admissible epimorphism \tikzcs{g: P \arrow[r, two heads] \& X} with $P$ an \calS-projective object, and {\bf enough \calS-injectives} if for all objects $X$  there exists an admissible monomorphism \tikzcs{f: X \arrow[r, tail] \& I} with $I$ an \calS-injective object.

\begin{definition}
An exact category $(\F, \calS)$ is {\bf Frobenius} if it has enough \calS-projectives, enough \calS-injectives and if $\proj \F = \inj \F$, where \proj \F and \inj \F denote the collection of $\calS$-projective and $\calS$-injective objects, respectively. The {\bf stable category} associated with \F is the quotient category $\uF:=\F / \inj \F$. We use $\uHom(X, Y)$ instead of $\Hom_{\uF}(X, Y)$ to denote the set of morphisms from $X$ to $Y$ in \uF. The equivalence class of a morphism $f: X \to Y$ in \F will be denoted by $\underline f$ in \uF. 
\end{definition}

We assume throughout the rest of this section that our exact category \E is Frobenius.

\begin{remark} \label{rem:factors through I(X)}
Assume that $f: X \to Y$ factors through some \calS-injective object $J$ as in the diagram below. 
Let \tikzcs{\mu : X \arrow[r, tail]\& I} be an \amono, $I$ an \calS-injective. Then there exists $\alpha: I \to Y$ such that $\alpha \mu = f$.
\begin{displaymath}
  \begin{tikzcd}
		X \arrow[rr, "f"] \arrow[dr, "\mu", tail] \arrow[ddr, "g"'] &&  Y \\
      & I \arrow[ru, "\alpha", dotted] \arrow[d, "\beta", dashed] &\\
      & J \arrow[ruu, "h"']
  \end{tikzcd}
\end{displaymath}
Namely, since $J$ is an \calS-injective object and $\mu$ is an \amono, there exists $\beta:~I\to~J$ such that $\beta \mu = g$. Let $\alpha := h \beta$. Then $\alpha \mu = h \beta \mu = h g = f$. 
\end{remark}


The stable category \uF is a triangulated category. 
The autoequivalence $T: \uF \to \uF$ is constructed by choosing, for all objects $X$, a sequence \tikzcs{X \arrow[r, tail]  \&  I \arrow[r, two heads]  \&  X'} in \calS for with $I$ an \calS-injective object and defining $TX:= X'$. This assignment is functorial in the stable category: given a morphism $f: X \to Y$ there exist morphisms such that the following diagram commutes
\begin{displaymath}
  \begin{tikzcd} 
	X \arrow[r, "\mu", tail]  \arrow[d,"f"] &  I \arrow[r, "\pi", two heads]  \arrow[d, "I(f)", dashed] &  X'  \arrow[d, "f'", dashed]\\
	Y \arrow[r, "\mu'", tail] &  I' \arrow[r, "\pi'", two heads]  &  Y'  
  \end{tikzcd}
\end{displaymath}
We define $T(f) := f'$: and it is not hard to see that $\udl{T(f)}$ is uniquely determined in \uF.

Now take a morphism $f :X\to Y$  in \F and consider the following diagram
\begin{displaymath}
  \begin{tikzcd}
 	X \arrow[r, "\mu", tail] \arrow[d, "f"] \arrow[dr, phantom, "\scriptscriptstyle PO" ] & I \arrow[r, "\pi", two heads] \arrow[d, "\overline{f}"]  & TX   \arrow[d, equal] \\
	 Y \arrow[r, "g", tail]  &  C_f \arrow[r, "h", two heads, dashed] & TX  
  \end{tikzcd}
\end{displaymath}
where $C_f$ is the pushout along $\mu$ and $f$. By \cite[Proposition 2.12]{Buhler} there exists a morphism $h: C_f \to TX$ such that the diagram commutes and $(g, h) \in \calS$. We call the triangle $X \xrightarrow{f} Y\xrightarrow{g} C_f \xrightarrow{h} TX$ and its image in \uF a {\bf standard triangle}, and we denote by $\Delta$ the collection of all triangles in \uF which are isomorphic to a standard triangle. The triple $(\E, T, \Delta)$ is then a triangulated category.

\section{Classifying subcategories}

As in the previous section, we fix an exact category $(\E, \calS)$. Furthermore, we fix a full subcategory \N which is closed under finite direct sums, and let $\E / \N$ be the quotient category. We denote morphisms in \EN by $\udl{f}$ (the equivalence class of a morphism $f$ in \E). Furthermore, denote by $\calS_\N$ the collection of all sequences $X \xrightarrow{\udl{f}} Y \xrightarrow{\udl{g}} Z$ in \EN for which there exists an isomorphism of sequences
\begin{displaymath}
  \begin{tikzcd}
	X  \arrow[r, "\udl{f}"] \arrow[d, "\cong"] \arrow[dr, phantom, "\circlearrowleft" ] &  Y \arrow[r, "\udl{g}"] \arrow[d, "\cong"]  \arrow[dr, phantom, "\circlearrowleft" ] &  Z \arrow[d, "\cong"]\\
	X'  \arrow[r, "\udl{f}'"] &  Y' \arrow[r, "\udl{g}'"] &  Z'
  \end{tikzcd}
\end{displaymath}
with \tikzcs{ X' \arrow[r, "f'", tail] \& Y' \arrow[r, "g'", two heads] \& Z'} $\in \calS$.

We shall establish a bijection between certain subcategories of the exact category \E  and the quotient category $\EN$.

\begin{definition} \label{def:complete and thick subcat}
A nonempty subcategory $\D$ of \E (respectively, of \EN) is a {\bf complete subcategory} if the following hold
\begin{itemize}
	\item[(i)] \D is a full subcategory, 
	\item[(ii)] 2 out of 3: If \tikzcs{X \arrow[r, "f", tail] \&  Y \arrow[r, "g", two heads]  \&  Z}  in \calS (respectively, in $\calS_\N$) and two of  $X, Y, Z$ are in \D, then so is the third.
\end{itemize}
Moreover, $\D$ is a {\bf thick subcategory} if in addition the following holds
\begin{itemize}
	\item[(iii)] \D is closed under direct summands, i.e. if $A$ is a direct summand of $X$, then $X \in \D$ implies $A\in \D$.
\end{itemize}
\end{definition}

\begin{remark} Let \D be a complete subcategory of \E. Take $X\in\D$: such an object exists since \D is nonempty. Then $0\in\D$ since 
	\tikzcs{X \arrow[r, "1_X", tail] \&  X \arrow[r, two heads]  \&  0}
is short exact. Hence by Proposition~\ref{prop:exact subcat is exact} \D is additive, closed under isomorphisms and admits an exact structure induced by the exact structure on \E.

By Remark~\ref{rem ex cat}~(1) all sequences of the form 
\begin{tikzcd}[ampersand replacement=\&]
	A \ar[r, "{\bmtxo{ 1 & 0 }}^t", tail] \& A \oplus B \ar[r, "{\bmtxo{ 0 & 1 }}", two heads] \& B
\end{tikzcd}
are short exact. Hence if two of $A$, $B$ and $A\oplus B$ are in \D, then so is the third by axiom (ii). We will use this fact repeatedly without referring to Remark~\ref{rem ex cat} every time.
\end{remark}


Our goal is to establish a one-to-one correspondence between the complete/thick subcategories of \E containing \N and the complete/thick subcategories of \EN, under the right assumptions. First we need to prove that complete and thick subcategories of \E containing \N are closed under isomorphisms when we pass to \EN. This is always true in the case of a thick subcategory, as the corollary of the next proposition shows. However, we need some extra assumptions in the case of a complete subcategory. 

\begin{proposition} \label{prop:Em closed under direct summands in EN}
Let \D be a thick subcategory of \E containing \N. If $Y$ is a direct summand of $X$ in the quotient category \EN, then $X \in \D$ implies $Y \in \D$.
\end{proposition}

\begin{proof}
Assume that we have $Y \xrightarrow{\udl{g}} X \xrightarrow{\udl{f}} Y$ with $\udl{fg} = \udl{1}_Y$ in \EN. Then there exists a commutative diagram in \E
\begin{displaymath}
  \begin{tikzcd}[sep=small]
	Y \arrow[rr, "fg-1_Y"]  \arrow[dr,"\alpha"'] &  &  Y\\
	&  N \arrow[ur, "\beta"']  &
  \end{tikzcd}
\end{displaymath}
with $N \in \N$. Now define morphisms $i, p$ by
\begin{displaymath}
\begin{tikzcd}[ampersand replacement=\&, sep = huge]
	Y \ar[r, "i:=\bmtx{g \\ \alpha}"] \&  X \oplus N \ar[r, "p:=\bmtx{f & -\beta}"] \& Y \end{tikzcd}
\end{displaymath}
Then $p i = \bmtx{f & -\beta} \bmtx{g \\ \alpha} = fg - \beta \alpha = 1_Y$, so $Y$ is a direct summand of $X \oplus N$. Moreover, since $X, N \in \D$ the 2 out of 3 property gives that $X \oplus N \in \D$. Hence $Y$ is a direct summand of an object in \D, implying $Y\in \D$.
\end{proof}

\begin{corollary} \label{cor:Em closed under iso in EN if thick}
Let \D be a thick subcategory of \E containing \N. Then \D is closed under isomorphisms in the quotient category $\E / \N$: if $X$ and $Y$ are isomorphic in $\E / \N$, then $X \in \D$ implies $Y \in \D$. 
\end{corollary}

We now state the additional assumptions we need to prove the above for complete (not necessarily thick) subcategories containing \N. 

\begin{definition}
A (full) subcategory \N of \E is {\bf factorization admissible} if it is closed under finite direct sums, direct summands and if whenever a morphism $f: X \to Y$ factors through an object in \N, then there exists a factorization
\begin{displaymath}
  \begin{tikzcd}[sep=small]
	X \arrow[rr, "f"]  \arrow[dr,"\alpha"'] &  &  Y\\
	&  N \arrow[ur, "\beta"']  &
  \end{tikzcd}
\end{displaymath}
with $N \in \N$ such that either $\alpha$ is an \amono or $\beta$ is an \aepi.
\end{definition}

As the following example shows, such subcategories appear naturally. 

\begin{examples} \label{ex:inj E factorization admissible}
If \E has enough \calS-injective objects, then \inj \E is a factorization admissible subcategory. Indeed, we may take $\alpha$ to be an \amono \tikzcs{\mu: X \arrow[r, tail] \& I} as described in Remark~\ref{rem:factors through I(X)}. Moreover, \inj \E is closed under finite direct sums and direct summands. 
Similarly if  \E has enough \calS-projective objects, then \proj \E is a factorization admissible subcategory of \E.
\end{examples}

The other technical condition we need when we deal with complete subcategories is an adapted version of the Five Lemma. 

\begin{definition}
We say that the quotient category \EN satisfies the { \bf Weak Five Lemma} for $\calS_\N$ if whenever we have a morphism of sequences in $\calS_\N$
\begin{displaymath}
  \begin{tikzcd}
	& X  \arrow[r] \arrow[d, "\udl{f}"] &  Y \arrow[r] \arrow[d, "\udl{g}"]  &  Z \arrow[d, "\udl{h}"] & \in \calS_\N\\
	& X'  \arrow[r] &  Y' \arrow[r] &  Z'  & \in \calS_\N
  \end{tikzcd}
\end{displaymath}
then the following hold:
\begin{itemize}
\item[(i)] If $\udl{f}$ and $\udl{g}$ are isomorphisms and $Z'=0$, then $\udl{h}$ is an isomorphism (i.e. $Z \cong 0$). 
\item[(ii)] Dually, if $\udl{g}$ and $\udl{h}$ are isomorphisms and $X=0$, then $\udl{f}$ is an isomorphism. 
\end{itemize}
\end{definition}

\begin{lemma} \label{lem:iso to 0 in quot}
If $X \cong 0$ in $\uE$, then $X$ is a direct summand of an object in \N.
\end{lemma}

\begin{proof} 
We have $\udl{1}_X = \udl{0}$ in \EN, so there is a factorization
\begin{displaymath}
  \begin{tikzcd}[sep=small]
	X \arrow[rr, "1_X-0"]  \arrow[dr,"\alpha"'] &  &  X\\
	&  N \arrow[ur, "\beta"']  &
  \end{tikzcd}
\end{displaymath}
in \E with $N \in \N$. Hence $\beta \alpha = 1_X$, so $X$ is a direct summand of $N$.
\end{proof}

We now prove a version of Corollary~\ref{cor:Em closed under iso in EN if thick} for complete subcategories.

\begin{proposition} \label{prop:Closed under iso in the quotient cat if N factorization admissible}
Assume that \N is a factorization admissible subcategory of \E and that \EN satisfies the Weak Five Lemma for $\calS_\N$. If \D is a complete subcategory of \E containing \N, then \D is closed under isomorphisms in the quotient category \uE. 
\end{proposition}

\begin{proof}
Assume that $X \in \D$ is isomorphic to $Y$ in \uE, and that $X \overset{\udl{f}}{\underset{\udl{g}}{\rightleftarrows}} Y$ are inverse isomorphisms. Then $fg-1_Y$ factors through some object $N \in \N$ as in the diagram below. 
\begin{displaymath}
  \begin{tikzcd}[sep=small]
	Y \arrow[rr, "fg-1_Y"]  \arrow[dr,"\alpha"'] &  &  Y\\
	&  N \arrow[ur, "\beta"']  &
  \end{tikzcd}
\end{displaymath}
We may assume that either $\alpha$ is an \amono or that $\beta$ is an \aepi. If $\alpha$ is an \amono, consider the \PO 
\begin{displaymath}
  \begin{tikzcd}
	 Y \arrow[r, "g"] \arrow[d, "\alpha"', tail] \arrow[dr, phantom, "\scriptscriptstyle PO" ] &  X \arrow[d, "a"] \\
	 N \arrow[r, "g'"] &  C
  \end{tikzcd}
\end{displaymath}
By \cite[Proposition 2.12]{Buhler} we get that 
\begin{displaymath}
\begin{tikzcd}[ampersand replacement=\&, sep = large]
	Y \ar[r, "{\bmtx{ g \\ \alpha }}", tail] \& X \oplus N \ar[r, "{\bmtx{ a & -g' }}", two heads] 
	\& C 
\end{tikzcd}
\end{displaymath}
lies in \calS. Hence we get the following morphism of sequences in $\calS_\N$
\begin{displaymath}
  \begin{tikzcd}[ampersand replacement=\&, sep = large]
	Y \arrow[r, "\udl{\bmtxt{ g \\ \alpha }}"] \arrow[d, "\udl{g}"] \arrow[dr, phantom, "\circlearrowleft" ] \& X \oplus N \arrow[r, "\udl{\bmtxo{ a & -g' }}"] \arrow[d, "\udl{\bmtxo{ 1 & 0}}"] \arrow[dr, phantom, "\circlearrowleft" ] \& C \arrow[d] \& \in \calS_\N \\
	X  \arrow[r, "\udl{1}_X"] \&  X \arrow[r] \&  0 \& \in \calS_\N
  \end{tikzcd}
\end{displaymath}
Since $\udl{g}$ and $\udl{\bmtxo{ 1 & 0}}$ are isomorphisms in \EN, so is the map $C \to 0$ by the Weak Five Lemma for $\calS_\N$. By Lemma~\ref{lem:iso to 0 in quot}, $C \cong 0$ implies that $C \in \N$ since \N is closed under direct summands. Furthermore, \D contains \N, so this gives $C \in \D$. We have $X \oplus N \in \D$ by the 2 out of 3 property, since $X, N \in \D$. Hence the 2 out of 3 property applied to \tikzcs{Y \arrow[r, tail] \& X \oplus N \arrow[r, two heads] \& C} implies that $Y \in \D$.

If instead $\beta$ is an \aepi, consider the \PB 
\begin{displaymath}
  \begin{tikzcd}
	 C \arrow[r, "b"] \arrow[d, "f'"'] \arrow[dr, phantom, "\scriptscriptstyle PB" ] &  X \arrow[d, "f"] \\
	 N \arrow[r, "\beta", two heads] &  Y
  \end{tikzcd}
\end{displaymath}
We get
\begin{tikzcd}[ampersand replacement=\&, sep = large, cramped]
	C \ar[r, "{\bmtx{ b \\ f' }}", tail] \& X \oplus N \ar[r, "{\bmtx{ f & -\beta }}", two heads] \& Y
\end{tikzcd}
$\in \calS$ and a morphism of sequences in $\calS_\N$
\begin{displaymath}
  \begin{tikzcd}[ampersand replacement=\&, sep = large]
	\& 0  \arrow[r] \arrow[d] \&  X \arrow[r, "\udl{1}_X"] \arrow[d, "\udl{\bmtxo{ 1 \\ 0}}"'] \&  X \arrow[d, "\udl{f}"] \& \in \calS_\N\\
	\& C \arrow[r, "\udl{\bmtxt{ b \\ f' }}"]  \& X \oplus N \arrow[r, "\udl{\bmtxo{ f & -\beta  }}"] \& Y \& \in \calS_\N 
  \end{tikzcd}
\end{displaymath}
As before, the fact that $\udl{f}$ and $\udl{\bmtxt{ 1 \\ 0}}$ are isomorphisms in \EN implies that $C \in \N \subseteq \D$, which again implies that $Y \in \D$.
\end{proof}



%

%
%
%

Recall that we have fixed a full subcategory \N of our exact category \E, and that this subcategory is closed under finite direct sums. We assume that \N is factorization admissible only when stated. We now construct the maps 
\begin{displaymath}
\begin{tikzcd}
	\{\text{complete subcategories of \E containing \N}\} \ar[r, "F"] &  \{\text{subcategories of \uE} \}
\end{tikzcd}
\end{displaymath}
\begin{displaymath}
\begin{tikzcd}
	\{\text{subcategories of \E containing \N}\} & \ar[l, "G"'] \{\text{complete subcategories of \uE} \}
\end{tikzcd}
\end{displaymath}
that will induce inverse bijections in several cases.

\defi{
{\bf (a) } For a complete subcategory \D of \E containing \N, define $F\D$ to be the full subcategory of \uE whose objects are the objects of \D.

{\bf (b) } For a complete subcategory \uEm of \uE, define $G\uEm$ to be the full subcategory of \E whose objects are the objects in \uEm (including \N). 
}

We first prove that $F$ and $G$ are maps between the two collections of complete subcategories, under the right assumptions. 

\begin{proposition} \label{prop: complete gives complete (if closed under iso)}
{\bf (1) } If \uEm is a complete subcategory of \uE, then $G\uEm$ is a complete subcategory of \E containing \N. Moreover, if \uEm is thick, then so is $G\uEm$. 

{\bf (2) } Assume that \D is closed under isomorphisms in \uE. If \D is a complete subcategory of \E (containing \N), then $F\D$ is a complete subcategory of \uE.
\end{proposition}

\begin{proof}
{\it (1):} Assume that \uEm is a complete subcategory of \uE and consider the subcategory $G\uEm$ of \E. Note that $G\uEm$ is full by definition. Assume that \tikzcs{ X \arrow[r, "f", tail] \&  Y \arrow[r, "g", two heads]  \&  Z} is in \calS. Then \tikzcs{X \arrow[r, "\udl{f}"] \&  Y \arrow[r, "\udl{g}"]  \&  Z} is in $\calS_\N$. Hence the 2 out of 3 property of $G\uEm$ follows from 2 out of 3 property of \uEm since $\obj G\uEm = \obj\uEm$. Furthermore, if $A$ is a direct summand of $X$ in \E, then $A$ is also a direct summand of $X$ in \EN. Hence $G\uEm$ is closed under direct summands in \E since $\uEm$ is closed under direct summands in \uE. Note that all objects in \N are isomorphic to the zero object in \uE. Since \uEm is closed under isomorphisms and $0\in \uEm$, this implies that \N is contained in \uEm, and therefore also in $G\uEm$.

{\it (2):} Assume that \D is a complete subcategory of \E (containing \N), closed under isomorphisms in \uE, and consider the subcategory $F\D$ of \uE. As above, $F\D$  is full by definition. Furthermore, assume that $X \xrightarrow{\udl{f}} Y \xrightarrow{\udl{g}} Z \in \calS_\N$. By the definition of $\calS_\N$ there exists an isomorphism
\begin{displaymath}
  \begin{tikzcd}
	X  \arrow[r, "\udl{f}"] \arrow[d, "\cong"] &  Y \arrow[r, "\udl{g}"] \arrow[d, "\cong"]  &  Z \arrow[d, "\cong"]\\
	X'  \arrow[r, "\udl{f}'"] &  Y' \arrow[r, "\udl{g}'"] &  Z'
  \end{tikzcd}
\end{displaymath}
of sequences in \uE with \tikzcs{ X' \arrow[r, "f'", tail] \& Y' \arrow[r, "g'", two heads] \& Z'} short exact in \E. Since $\D$ is closed under isomorphisms in \uE, the 2 out of 3 property of $F\D$ follows directly from the 2 out of 3 property of \D. Indeed, assume for example that $X, Y \in \obj F\D = \obj \D$. Then $X', Y' \in \obj \D$ since \D is closed under isomorphisms in \uE. Moreover, since the sequence \tikzcs{ X' \arrow[r, "f'", tail] \& Y' \arrow[r, "g'", two heads] \& Z'} short exact in \E, the 2 out of 3 property of \D implies that $Z' \in \obj \D$. Again, \D is closed under isomorphisms in \uE, so this implies that \mbox{ $Z \in \obj \D = \obj F\D$.} 
\end{proof}

\begin{corollary}
{\bf (1) } Assume that \N is a factorization admissible subcategory of \E and that \EN satisfies the Weak Five Lemma for $\calS_\N$. If \D is a complete subcategory of \E containing \N, then $F\D$ is a complete subcategory of \uE.

{\bf (2) } If \D contains \N and is thick, then so is $F\D$.
\end{corollary}

\begin{proof} 
The subcategory \D is closed under isomorphisms in \uE in both cases by Proposition~\ref{prop:Closed under iso in the quotient cat if N factorization admissible} and Corollary~\ref{cor:Em closed under iso in EN if thick}, respectively. Thus by Proposition~\ref{prop: complete gives complete (if closed under iso)} $F\D$ is a complete subcategory of \uE. Moreover, if \D is closed under direct summands in \E, then $F\D$ is closed under direct summands in \EN by Proposition~\ref{prop:Em closed under direct summands in EN}.
\end{proof}

We now have the following main result, which follows from the above. 

\begin{theorem} \label{thm:Main thm}
Let \E be an exact category and \N a full subcategory closed under finite direct sums.

{\bf (1) } There is a one-to-one correspondence between thick subcategories of \E containing \N and thick subcategories of the quotient category \uE given by
\begin{displaymath}
  \begin{tikzcd}
	\{\text{thick subcategories of \E containing \N}\} \ar[r, bend left, "F", start anchor=north east, end anchor=north west] & \ar[l, bend left, "G", start anchor=south west, end anchor=south east] \{\text{thick subcategories of \uE} \}.
 \end{tikzcd} 
\end{displaymath}

{\bf (2) } Assume in addition that \N is factorization admissible and that \EN satisfies the Weak Five Lemma for $\calS_\N$. Then there is a one-to-one correspondence between complete subcategories of \E containing \N and complete subcategories of the quotient category \uE given by
\begin{displaymath}
  \begin{tikzcd}
	\{\text{complete subcategories of \E containing \N}\} \ar[r, bend left, "F", start anchor=north east, end anchor=north west] & \ar[l, bend left, "G", start anchor=south west, end anchor=south east] \{\text{complete subcategories of \uE} \}. 
  \end{tikzcd}
\end{displaymath}
\end{theorem}


Next we consider the special case where $(\F, \calS)$ is Frobenius and $\N = \inj \F$, i.e. the quotient category $\F/\N$ is the stable category \uF. Note that \inj \F is a factorization admissible subcategory of \F as described in Example~\ref{ex:inj E factorization admissible}. 

\begin{lemma} \label{lem:in calS-N iff in Delta}
A sequence $X \xrightarrow{\udl{f}} Y \xrightarrow{\udl{g}} Z$ is in $\calS_\N$ if and only if there exists a morphism $\udl{h}: Z \to TX$ such that $X \xrightarrow{\udl{f}} Y \xrightarrow{\udl{g}} Z \xrightarrow{\udl{h}} TX$ is a distinguished triangle in \E.
\end{lemma}

\begin{proof}
Assume that $X \xrightarrow{\udl{f}} Y \xrightarrow{\udl{g}} Z \in \calS_\N$. Then there exists an isomorphism of sequences given by the solid part of the diagram
\begin{displaymath}
  \begin{tikzcd}
	X  \arrow[r, "\udl{f}"] \arrow[d, "\cong", "\udl{\varphi_1}"'] &  Y \arrow[r, "\udl{g}"] \arrow[d, "\cong", "\udl{\varphi_2}"']  &  Z \arrow[d, "\cong", "\udl{\varphi_3}"'] \arrow[rr, "\udl{T\varphi_1^{-1} h' \varphi_3}", dotted] && TX \arrow[d, "\cong", "T\udl{\varphi_1}"']\\
	X'  \arrow[r, "\udl{f}'"] &  Y' \arrow[r, "\udl{g}'"] &  Z' \arrow[rr, "\udl{h}'", dashed] && TX'
  \end{tikzcd}
\end{displaymath}
with \tikzcs{ X' \arrow[r, "f'", tail] \& Y' \arrow[r, "g'", two heads] \& Z'} $\in \calS$. By \cite[Chapter I, section 2.7]{Happel} there exists in \E a morphism 
	$h': Z' \to TX'$ 
 such that 
	$X' \xrightarrow{\udl{f}'} Y' \xrightarrow{\udl{g}'} Z' \xrightarrow{\udl{h}'} TX' \in \Delta$ (where $\Delta$ is the collection of distinguished triangles in \uF). 
Define 
	$\udl{h} := (T\udl{\varphi_1})^{-1} \circ \udl{h}' \circ \udl{\varphi_3}$. 
Then 
	$(\varphi_1, \varphi_2, \varphi_3)$ 
is an isomorphism between the triangles 
	$X' \xrightarrow{\udl{f}'} Y' \xrightarrow{\udl{g}'} Z' \xrightarrow{\udl{h}'} TX' \in \Delta$ and $X \xrightarrow{\udl{f}} Y \xrightarrow{\udl{g}} Z \xrightarrow{\udl{h}} TX$. 
Hence the latter triangle is also in $\Delta$.

Now assume that 
	$X \xrightarrow{\udl{f}} Y \xrightarrow{\udl{g}} Z \xrightarrow{\udl{h}} TX \in \Delta$. 
Then it is fairly easy to show that there exist 
	\tikzcs{ X' \arrow[r, "f'", tail] \& Y' \arrow[r, "g'", two heads] \& Z'}
	$\in \calS$ 
and a morphism $h': Z' \to TX'$ such that 
	$X \xrightarrow{\udl{f}} Y \xrightarrow{\udl{g}} Z \xrightarrow{\udl{h}} TX$ and $X' \xrightarrow{\udl{f}'} Y' \xrightarrow{\udl{g}'} Z' \xrightarrow{\udl{h}'} TX'$ 
are isomorphic as triangles. This isomorphism of triangles clearly gives an isomorphism of the sequences 
	$X \xrightarrow{\udl{f}} Y \xrightarrow{\udl{g}} Z$ and $X' \xrightarrow{\udl{f}'} Y' \xrightarrow{\udl{g}'} Z' \in \calS_\N$, hence $X \xrightarrow{\udl{f}} Y \xrightarrow{\udl{g}} Z  \in \calS_\N$.
\end{proof}

Recall that a nonempty full subcategory $\T'$ of a triangulated category $(\T, T, \Delta)$ is a {\bf triangulated subcategory} if the 2 out of 3 property holds. Namely if $X \xrightarrow{f} Y \xrightarrow{g} Z \xrightarrow{h} TX$ is a distinguished triangle in $\T$ and two of the objects $X, Y, Z$ are in $\T'$, then so is the third. Moreover, $\T'$ is a {\bf thick triangulated subcategory} if in addition it is closed under direct summands.

Note that it follows from Lemma~\ref{lem:in calS-N iff in Delta} that a subcategory \uFm of \uF is a triangulated subcategory (respectively, thick triangulated subcategory) if and only if it is a complete subcategory (respectively, thick subcategory) as in the sense of Definition~\ref{def:complete and thick subcat}.
Note also that the Five Lemma for triangulated categories implies that \uF satisfies the Weak Five Lemma for $\calS_\N$.


We are now ready to state Theorem~\ref{thm:Main thm} in the case of a Frobenius category and the associated stable category. The result follows from the above.

\begin{theorem} \label{thm:Main thm Frob}
Let \F be a Frobenius category. 

{\bf (1) } There is a one-to-one correspondence between thick subcategories of \F containing \inj \F and thick triangulated subcategories of the associated stable category \uF given by
\begin{displaymath}
  \begin{tikzcd}
	\{\text{thick subcategories of \F containing \inj \F}\} \ar[r, bend left, "F", start anchor=north east, end anchor=north west] & \ar[l, bend left, "G", start anchor=south west, end anchor=south east] \{\text{thick triangulated subcategories of \uF} \}.
  \end{tikzcd} 
\end{displaymath}

{\bf (2) } There is a one-to-one correspondence between complete subcategories of \F containing \inj \F and triangulated subcategories of the associated stable category \uF given by
\begin{displaymath}
  \begin{tikzcd}
	\{\text{complete subcategories of \F containing \inj \F}\} \ar[r, bend left, "F", start anchor=north east, end anchor=north west] & \ar[l, bend left, "G", start anchor=south west, end anchor=south east] \{\text{triangulated subcategories of \uF} \}.
  \end{tikzcd} 
\end{displaymath}
\end{theorem}

\section{The stable category of Gorenstein projective objects of an abelian category}

Fix an abelian category \A with enough projective objects. An object $X$ in \A is a {\bf Gorenstein projective object} if there exists a complex 
	$$\mathbf{P} = \ \ \ (\dots \to P_2 \to P_1 \to P_0 \xrightarrow{f} P_{-1} \to P_{-2} \to \dots)$$
of projective objects in \A, such that $X = \Image f$ and both $\mathbf{P}$ and $\Hom_\A(\mathbf{P}, Q)$ are acyclic for all $Q \in \proj \A$. In this case $\mathbf{P}$ is called a {\bf complete projective resolution} of $X$. We denote by $\Gproj \A$ the full subcategory of \A consisting of all Gorenstein projective objects in \A. By \cite{Beli}, $\Gproj \A$ is a Frobenius category, and the projective-injective objects in $\Gproj \A$ are the projective objects in \A. 

From Theorem~\ref{thm:Main thm Frob} and the above we deduce the following result.

\begin{theorem} \label{thm:Main Gproj}
Let \A be an abelian category with enough projectives, and let $\udl{\Gproj}$ \A denote the stable category associated to the Frobenius category \Gproj \A. 
Then we have the following one-to-one correspondences 
\begin{displaymath}
    \begin{tikzcd}
	\{\text{thick subcategories of \Gproj \A containing \proj \A}\} \ar[d, leftrightarrow, "1-1"] \\  \{\text{thick triangulated subcategories of $\udl{\Gproj}$ \A} \}
  \end{tikzcd} 
\end{displaymath}

\begin{displaymath}
  \begin{tikzcd}
	\{\text{complete subcategories of \Gproj \A containing \proj \A}\} \ar[d, leftrightarrow, "1-1"] \\ \{\text{triangulated subcategories of $\udl{\Gproj}$ \A} \}
  \end{tikzcd}
\end{displaymath}
\end{theorem}

Now let $R$ be a commutative Noetherian ring and denote by $\modR$ the (abelian) category of all finitely generated $R$-modules. To simplify notation we define $\Gproj R := \Gproj \modR$ and $\proj R := \proj \modR$. Recall that a module $M \in \modR$ is {\bf totally reflexive} if the following hold:
\begin{itemize}
	\item[(i)] the natural biduality map $M \to M^{**}$ is an isomorphism,
	\item[(ii)] $\Ext_R^{>0}(M, R) = 0$,
	\item[(iii)] $\Ext_R^{>0}(M^{*}, R) = 0 $,
\end{itemize}
where $M^* := \Hom(M,R)$ denotes the dual module of $M$.
We denote by \GR the full subcategory of \modR consisting of all totally reflexive modules. It is well-known that $\GR = \Gproj R$, thus \GR is Frobenius with the projective $R$-modules as the projective-injective objects. 

Note that if a subcategory \D of \GR is closed under finite direct sums and direct summands, then it contains $\proj R$ if and only if it contains the object $R$. Thus the condition that a thick subcategory must contain $\proj R$ can be replaced by the condition that it must contain the object $R$. Consequently, in this setting, Theorem~\ref{thm:Main Gproj} takes the following form.

\begin{theorem} \label{thm:Main tot ref}
Let $R$ be a commutative Noetherian ring, \GR the Frobenius category of totally reflexive modules, and 
$\udl{\mathcal{G}} (R)$ the corresponding stable category. 
Then we have the following one-to-one correspondences 
\begin{displaymath}
    \begin{tikzcd}
	\{\text{thick subcategories of \GR containing $R$}\} \ar[d, leftrightarrow, "1-1"] \\  \{\text{thick triangulated subcategories of $\udl{\mathcal{G}} (R)$} \}
  \end{tikzcd}
\end{displaymath}
\begin{displaymath}
  \begin{tikzcd}
	\{\text{complete subcategories of \GR containing $\proj R$}\} \ar[d, leftrightarrow, "1-1"] \\ \{\text{triangulated subcategories of $\udl{\mathcal{G}} (R)$} \}
  \end{tikzcd}
\end{displaymath}
\end{theorem}

Specializing further, let now $(R, \m, k)$ be a commutative Noetherian local ring. Recall that the depth of an $R$-module $M$ is 
	$$\depth M = \inf \{ n \geq 0 \ | \ \Ext^n _R (k, M) \neq 0\},$$
and $M$ is a maximal Cohen-Macaulay module if $M=0$ or $\depth M = \dim R$ (where $\dim R$ denotes the Krull dimention of $R$). We denote by $\CM(R)$ the full subcategory of $\mmod R$ consisting of all maximal Cohen-Macaulay modules over $R$. 
	
The ring $R$ is a Gorenstein local ring if it has finite injective dimension as a module over itself. The maximal Cohen-Macaulay modules over a Gorenstein local ring are precisely the totally reflexive modules (cf.\ \cite{Buch} or \cite[Theorem 4.8]{Ryo}). Thus $\CM(R) = \GR = \Gproj R$, and $\CM(R)$ is a Frobenius category. Consequently, in this setting, Theorem~\ref{thm:Main tot ref} takes the following form. The ``thick'' part of this result is known, c.f\ \cite{RyoCM}.

\begin{theorem} 
Let $R$ be a Gorenstein local ring, $\CM(R)$ the Frobenius category of maximal Cohen-Macaulay modules, and $\udl{\CM} (R)$ the corresponding stable category. 
Then we have the following one-to-one correspondences 
\begin{displaymath}
  \begin{tikzcd}
	\{\text{complete subcategories of $\CM(R)$ containing $\proj R$}\} \ar[d, leftrightarrow, "1-1"] \\ \{\text{triangulated subcategories of $\udl{\CM} (R)$} \}
  \end{tikzcd}
\end{displaymath}

\begin{displaymath}
    \begin{tikzcd}
	\{\text{thick subcategories of $\CM(R)$ containing $R$}\} \ar[d, leftrightarrow, "1-1"] \\  \{\text{thick triangulated subcategories of $\udl{\CM} (R)$} \}
  \end{tikzcd}
\end{displaymath}
\end{theorem}


%
%

\bibliographystyle{plain}

%
%
\end{document}